\documentclass[11pt,a4paper]{article}

\usepackage{amsmath, amsfonts, amssymb,mathtools,xcolor}
\usepackage[normalem]{ulem}
\usepackage{caption}
\usepackage{accents}
\usepackage{multirow}
\usepackage{comment}
\usepackage{ulem}
\usepackage{fullpage}
\usepackage{authblk}
\usepackage[hidelinks]{hyperref}
\usepackage{multicol}
\usepackage{amsthm}

\DeclareMathOperator*{\essinf}{ess\,inf}
\usepackage{bbm}

\newtheorem{prop}{Proposition}[section]

\newtheorem{lemma}{Lemma}[section]
\newtheorem{corollary}{Corollary}[section]
\newtheorem{assumption}{Assumption}[section]
\newtheorem{defi}{Definition}[section]
\newtheorem{rem}{Remark}[section]

\title{Pathwise stochastic control and a class of stochastic partial differential equations\\
}
 \author[1]{Neeraj Bhauryal}
 \author[2]{Ana Bela Cruzeiro}
 \author[3]{Carlos Oliveira}
 \affil[1]{Grupo de Física Matemática Univ. de Lisboa, Portugal}
\affil[2]{GFMUL and Dep. de Matemática, Instituto Superior Técnico, Lisboa, Portugal}
\affil[3]{\textit{Dep. of Industrial Economics and Technology Management, NTNU, Norway}
\\
\textit{and ISEG - School of Economics and Management, Uni. de Lisboa}\\
\textit{Research in Economics and Mathematics, CEMAPRE, Portugal}\\
}
\affil[1]{\textit{Email: nsbhauryal@fc.ul.pt}}
\affil[2]{\textit{Email: ana.cruzeiro@tecnico.ulisboa.pt}}
\affil[3]{\textit{Email: carloso.m.d.s.oliveira@ntnu.no}}

\begin{document}
\date{}
\maketitle
\makeatletter 
\renewcommand\theequation{\thesection.\arabic{equation}}
\@addtoreset{equation}{section}
\makeatother 

\begin{abstract}
We consider a pathwise stochastic optimal control problem and study the associated (not necessarily adapted) Hamilton-Jacobi-Bellman stochastic partial differential equation. We show that the value process is the unique solution of this equation, in the viscosity sense. Finally, when it is well defined, we discuss some properties of the optimal drift. \color{black}
\it  \rm

\end{abstract}

\vskip 7mm

\section{ Introduction}\label{section1.}
In classical stochastic control problems, one aims  to minimize (or maximize) some performance criterium, described by a functional of the form 
\begin{align*}
I_{t,x}(Z,u)&=E\left[\int_t^T f(s,Z_s,u_s)ds+g(Z_T)\right]\\
dZ_s&=\alpha(s,Z_s,u_s)ds+\sigma(s,Z_s,u_s)dW_s,\quad Z_t=x\text{ and }0\leq t\leq s\leq T,
\end{align*}
where $f,g,\alpha,\sigma$ are deterministic functions, and $u$ is a control in a given set $\cal U$. 
\color{black} One wants therefore to find the so-called value function $\Xi(t,x)=\inf_{u\in{\mathcal{U}}} I_{t,x}(Z,u)$, or equivalently the optimal control $u^*$ such that  $\Xi(t,x)=I_{t,x}(Z,u^*)$,
when the minimum exists. If the problem is formulated within a Markovian setup,  the control process is of the form $u^*\equiv u(s,Z_s)$, where $u$ is a deterministic function, and the value function solves a deterministic partial differential equation (PDE), namely the Hamilton-Jacobi-Bellman (HJB) equation (see, for instance, Chapter III of \cite{fleming2006controlled}).

This field has been growing in the last decades (see, for instance, \cite{fleming2006controlled,fleming2012deterministic,krylov2008controlled,yong1999stochastic}), mostly driven by its applications to Finance (\cite{pham2009continuous}), Insurance (\cite{schmidli2007stochastic}), Engineering (\cite{aastrom2012introduction,chen1995linear}), and Physics (\cite{nelson2021quantum}). In Finance and Insurance, stochastic control plays an important role in topics such as portfolio selection, optimal liquidation or in the definition of the optimal reinsurance strategies. In Physics and Engineering, the minimization of the previous action function when we choose $f(s,x,u)=\frac{1}{2}\vert u\vert^2+V(x)$, $\alpha(s,x,u)=u$, and $\sigma(s,x,u)=\sqrt{\nu}$ has great importance. {Indeed, this function $f$ is the Lagrangian of the well known least action principle in classical mechanics. This means that $I_{t,x}$ can be regarded as a regularisation of this deterministic action functional, where the time derivative of $Z_s$, which is now divergent, becomes the drift of the diffusion process.} 
The following result is described by Fleming and Soner \cite{fleming2006controlled}, in Example 8.2 of Chapter III.
\begin{prop}\label{prop deterministic HJB equation}
Let $v$ be a classical solution to the HJB equation and terminal condition
\begin{align}\label{deterministic HJB equation}
\frac{\partial v}{\partial t}(t,x)-\frac{1}{2}\vert\nabla v(t,x)\vert^2+\frac{\nu}{2}\Delta v(t,x)+V(x)=0\quad\text{and}\quad v(T,x)=S(x).
\end{align}
Then, $v(t,x)=\Xi(t,x)$, and $u^*(t,x)=-\nabla \Xi(t,x)$, when $f(s,x,u)=\frac{1}{2}\vert u\vert^2+V(x)$, $\alpha(s,x,u)=u$.
\end{prop}
In its turn, by using a change of variable, one can also see that the optimal control satisfies the following PDE:
$$
\frac{\partial u}{\partial t}(t,x)+(u\cdot\nabla)u(t,x)+\frac{\nu}{2}\Delta u(t,x)-\nabla V(x)=0\quad\text{and}\quad u(T,x)=-\nabla S(x).
$$

Recently, in \cite{cruzeiro2021time}, the authors have addressed this problem in a more general setup. They consider a backward and a forward action functional that can be controlled and stopped at any moment before time $T$. The new forward functional and control problem are given by 
\begin{align*}
\tilde I_{t,x}(Z,u,\tau)&=E\left[\int_t^{\tau\wedge T} \frac{1}{2}\vert u_s\vert^2+V(Z_s)ds+g(Z_ {\tau\wedge T})\right]\\
\tilde \Xi(t,x)&=\inf_{(u,\tau)\in{\cal U}\times{\cal T}}\tilde I_{t,x}(Z,u,\tau)
\end{align*}
The backward control problem can be formalized similarly considering a decreasing filtration (see Equations (3), (4) and (15) in \cite{cruzeiro2021time}). The solution for each one of these backward and forward control problems can be obtained as a viscosity solution to a free-boundary problem. We can state a similar result to Proposition \ref{prop deterministic HJB equation}, as a consequence of Proposition 4.4 and Theorem 5.1 in \cite{cruzeiro2021time}.\footnote{Conditions to guarantee existence of solution to the boundary problems \eqref{deterministic HJB equation} and \eqref{deterministic HJB equation 1} can be found in the respective references. }
\begin{prop}
Let $v$ be a classical solution to the HJB equation and terminal condition
\begin{align}\label{deterministic HJB equation 1}
\max\left\{-\frac{\partial v(t,x)}{\partial t}+\frac{1}{2}\vert\nabla v(t,x)\vert^2-\frac{\nu}{2}\Delta v(t,x)-V(x),v(t,x)-S(x)\right\}=0\quad\text{and}\quad v(T,x)=S(x).
\end{align}
Then, $v(t,x)=\tilde \Xi(t,x)$, $u^*(t,x)=-\nabla \tilde\Xi(t,x)$, and the optimal stopping time is given by $\tau^*=\inf\{s>t\,:\tilde \Xi(s,Z_s)\geq S(Z_s)\}$, when $f(s,x,u)=\frac{1}{2}\vert u\vert^2+V(x)$, $\alpha(s,x,u)=u$.
\end{prop}
In this case, one can easily see that the optimal control $u$ satisfies the boundary problem
\begin{align*}
&\frac{\partial u}{\partial t}(t,x)+(u\cdot\nabla)u(t,x)+\frac{\nu}{2}\Delta u(t,x)-\nabla V(x)=0,\quad (t,x)\in {\cal C}\\
&u(t,x)=-\nabla S(x),\quad (t,x)\in \partial{\cal C},
\end{align*}
where $\cal C$ is the so-called continuation region, which is defined as ${\cal C}=\{(t,x):\:\tilde \Xi(t,x)< S(x)\}$. The optimal stopping time can also be characterized in terms of cumulative distribution functions. Defining $q$ as $q(t,x)=P_{t,x}(\tau^*>\tilde T)$, we can obtain the following characterization:
\begin{align*}
&\frac{\partial q}{\partial t}+u^*\cdot\nabla q+\frac{\nu}{2}\Delta u=0,\quad (t,x)\in {\cal C}\\
&q(\tilde T,x)=1,\quad (\tilde T,x)\in {\cal C}\\
&q(\tilde t,\tilde x)=0,\quad (\tilde t,\tilde x)\in\partial{\cal C},
\end{align*}
where $t<\tilde{T}<\overline{t}$ and $\overline{t}=\sup\{t<T\,:(t,x)\in{\cal C}\}$. All these results require some regularity conditions that can be checked in \cite{cruzeiro2021time}.

In this paper, we  minimize a pathwise version of the previous functional, as in \eqref{random functional}. The value process is generally a non-adapted It\^o-type random field that can be obtained as a viscosity solution of a stochastic version of the HJB equation \eqref{deterministic HJB equation}, which is a stochastic partial differential equation (SPDE). We prove that the value process is the unique viscosity solution of this SPDE. When an optimal control drift exists, this drift also satisfies a SPDE. \color{black} 

Stochastic control problems with random functionals were first addressed by Lions and Souganidis in the collection of papers \cite{lions1998fully,lions2000fully,lions2000uniqueness}, where the authors propose a theory of stochastic viscosity solutions for the stochastic HJB equations that characterize the value process. Later, Buckdahn and Ma addressed the topic of stochastic viscosity solutions in \cite{buckdahn2001stochastic,BUCKDAHN2001205} using a different approach than of Lions and Souganidis.

A major difference in our setup is that solutions of the SPDE is not necessarily an adapted process, and, to our knowledge, equations of this type has not been studied before. The  kind of anticipativeness that has been considered in the literature appears (1) from choosing anticipative initial data (cf., for example \cite{Tusheng}) or (ii) in the context of backward SPDE's (cf., for example \cite{Pardoux}), but not from optimal control problems, as in our case. Let us consider a situation where solutions do not need to be adapted to a given filtration:
\begin{align*}
\begin{cases}
     du(t,x)= -\sqrt{\nu} \frac{\partial}{\partial x}u(t,x) dW(t)+\nu\left(\frac{\partial^2}{\partial^2 x} u+\frac{1}{2}|\frac{\partial}{\partial x} u|^2\right)\mathrm{d}t, \quad &\text{in } (0,T)\times \mathbb{R},\\
    u(T,x)=f(x), & \text{on } \mathbb{R}.
    \end{cases}
\end{align*}
for some $\nu>0$ and given terminal data. The explicit solution of this SPDE is given by $u(t,x)=\log \eta (t-T,x)$, where $\eta(t,x)$ is a positive solution of the stochastic heat equation
\begin{align*}
     d\eta(t,x)= -\sqrt{\nu}\frac{\partial}{\partial x} \eta(t,x) dW(t)+\nu \frac{\partial^2}{\partial^2 x}\eta(t,x)\mathrm{d}t, \quad &\text{in } (0,T)\times \mathbb{R}
\end{align*}
and it is given by $\eta(t,x)= \frac{1}{\sqrt{2\nu \pi t}}\int\limits_\mathbb{R} e^{-\frac{|x-y-\sqrt{\nu}W(t)|^2}{2\nu t}+f(y)}\mathrm{d}y$ (see Eq (2.3.9), page 53, of \cite{Lototsky}).

Buckdahn and Ma addressed a stepwise control problem, for which the value process is not adapted to the increasing filtration. The authors avoid using non-adapted stochastic calculus by constructing two auxiliary stochastic control problems that allow them to recover the solution of the original problem. Using the  Doss–Sussmann-type transformation, which they introduced previously in  \cite{buckdahn2001stochastic,BUCKDAHN2001205}, the authors prove that the value process is a unique viscosity solution to a certain SPDE with terminal data in \cite{buckdahn2007pathwise}.  In \cite{Burstein} a special kind of nonanticipative stochastic control problem is considered, by introducing nonanticipativity as a Lagrange multiplier. The approach reveals that pathwise control is in some sense equivalent to classical stochastic control with anticipative controls. The author in \cite{Rogers} studies a minimization of a stochastic action functional defined in discrete times by solving pathwisely deterministic control problems. The more recent work \cite{Cohen} considers also pathwise stochastic optimal control problems using rough path theory.

In Section \ref{non-adapted}, we recall the basic notions of non-adapted stochastic calculus, which will be used throughout the paper. 
 In Section \ref{section2.}, the control problem is introduced along with the set of assumptions required for our analysis.  In Section \ref{Stochastic HJB equation}, we establish the Bellman's optimality principle which allows us to prove the existence result and then we establish a comparison principle for viscosity solutions of the Stochastic HJB equation arising from the optimal control problem under consideration. Finally, in Section \ref{characterization}, a derivation of the SPDE for the optimal drift (when it is attained) is presented and  we discuss a concept of conserved quantities that should be relevant for our action functionals.
This section raises some new open problems, to be considered in the future, namely the well-posedness of the (non-adapted) SPDEs we have presented and the study of the conserved quantities.\color{black}

\section{Non-adapted stochastic calculus}\label{non-adapted}

From now on we fix, as our probability space, $\Omega =\{ \omega \in C( [t,T]; \mathbb R^n ), \omega (t)=x,  \omega ~\hbox{continuous} \}$, equipped with the topology of uniform convergence and with the $\sigma$-algebra generated by cylindrical sets. $P$ will be the standard Wiener measure. Denote by
$H$ the corresponding Cameron-Martin space, namely $H=\{ h:[t,T]\rightarrow \mathbb R^n : h(t)=0,  h ~\hbox{is a.c. and}~ \int_t^T |\frac{d}{ds} h (s)|^2 ds < +\infty \}$ with the inner product defined as $\langle g, h \rangle_H = \int^T_t \frac{\mathrm{d}g}{\mathrm{d}s}(s)\frac{\mathrm{d}h}{\mathrm{d}s}(s)\mathrm{d}s$.
If $ E |F|^p <+\infty$ for some $p\geq 1$, derivatives of $F$ in the directions of $H$ are defined, in the Malliavin calculus sense (cf. \cite{Malliavin}), as

$$D_h F (\omega )=\lim\limits_{\varepsilon \rightarrow 0} \frac{1}{\varepsilon} [F(\omega +\epsilon h )- F(w)],$$
the limit being taken almost surely (a.s.).
This derivative naturally gives rise, by Riesz representation theorem,  to a gradient operator $\boldsymbol\nabla F :\Omega \rightarrow H$ such that $\langle\boldsymbol{\nabla} F, h\rangle_H =D_h F$.
If we define
$$D_s F (\omega )= \frac{d}{ds} \boldsymbol{\nabla} F(\omega ),$$
we have $D_h F =\int_t^T D_s F \frac{d}{ds} h(s)ds$. For cylindrical functionals $F(w)= f(\omega (s_1 ) , ..., \omega (s_m ))$, with $f$ smooth, we have
$$D_s F(\omega )= \sum_{k=1}^m {\bf{1}}_{s<s_k} \partial_k f (\omega (s_1 ) , ..., \omega (s_m )).$$
The operator $D$ is a closed operator on the space $W_{1,2} (\Omega )$, the completion of cylindrical functionals with respect to the norm

$$\| F\|^2_{1,2}=E \Big( |F|^2 +\int_t^T |D_s F |^2 ds \Big).$$
In the non-adapted stochastic calculus developed by  Nualart and Pardoux (\cite{Nualart}), the It\^o-Skorohod integral of non necessarily adapted processes $\int u ~dW$ is defined as the $L^p $ limit,
 when it exists, of  sums
$$\sum_k {\cal M}_k (u)   (W (s_{k+1} )-W (s_k ))-\frac{1}{s_{k+1} -s_k} \int_{s_k}^{s_{k+1}}\int_{s_k}^{s_{k+1}} D_s u_{\tau} ds d\tau$$
where
$${\cal M} _k (u)= \frac{1}{s_{k+1} -s_k}  \int_{s_k}^{s_{k+1}} u_{\tau} d\tau,$$
when the mesh of the decomposition of the time interval goes to zero. It is an extension of the  It\^o integral.

We have the following commutation relation
$$D_s \int u~dW =\int D_s u_{\tau} dW_{\tau}  + u_s .$$
One can also define a Stratonovich-Skorohod integral $\int u \circ dW$ of non-adapted processes as the limit of sums

$$\sum_k {\cal M}_k (u)   (W (s_{k+1} )-W (s_k ))$$
As in the adapted case, Stratonovich integration obeys the rules of ordinary differential calculus. The relation between the two integrals is given by

$$\int_t^T u_s \circ dW_s =\int_t^T u_s dW_s +\frac{1}{2}\int_t^T  (\mathbb D u)_s ds,$$
where
$(\mathbb D u)_s = D_s^+ u_s + D_s^- u_s$, with
$$D_s^+ u_s =\hbox{lim}_{\tau \rightarrow s^+} D_s u_{\tau},~D_s^- u_s =\hbox{lim}_{\tau \rightarrow s^-} D_s u_{\tau},$$
the limit being taken in the $L^p$ sense. In the case where $u$ is adapted, $D_s^+ u_s =0$ and $\frac{1}{2}\int_t^T D_s^- u_s ds$ reduces to the usual It\^o contraction term.

We will need a It\^o-Wentzell formula for non-adapted stochastic integrals. Such a formula has been proved in \cite{Ocone}; we recall its Stratonovich version, which,
as long as we interpret the stochastic integrals in the sense of non-adapted calculus, looks formally similar to its adapted version.

Let $Z_s = Z_t + \int_t^s B_{\tau} \circ  dW_{\tau} +\int_t^s A_{\tau} d\tau$ and $F_s  (x) =F_t (x) +\int_t^s H_{\tau} (x) \circ dW_{\tau} +\int_t^s G_{\tau} (x)d\tau$. Then the following formula holds
\begin{align}\nonumber
F_s (Z_s )= F_t (Z_t ) &+\sum_k \int_t^s \nabla F_{\tau} (Z_{\tau}) \cdot B^k_{\tau} \circ dW^k_{\tau} +\int_t^s \nabla F_{\tau} (Z_{\tau}) \cdot A_{\tau}d\tau\\
&+ \int_t^s H_{\tau} (Z_{\tau}) \circ dW_{\tau} + \int_t^s G_{\tau} (Z_{\tau}) d\tau\label{Ito-Wentzell formula}
\end{align}
The above It\^o-Wentzell formula holds for non necessarily adapted stochastic processes $X$ and $F$ under a certain list of conditions on the coefficients of $X$ and
$F$ (cf. \cite{Ocone}). In our work it will only be applied to the process $Z$ as defined in \eqref{controlled diffusion} and to smooth and bounded (in space) functionals $F$, so our assumptions on the drift of $u$ are sufficient to ensure that the formula holds.



\section{Stochastic control problem }\label{section2.}
Let $Z\equiv\{Z_s\in\mathbb{R}^n:s\in I\equiv [0,T]\}$ be a stochastic process defined on the probability space specified in Section \ref{non-adapted}. Consider an energy function $L:\mathbb{R}^n \to \mathbb{R}$, a potential function $V:\mathbb{R}^n\to \mathbb{R}$ and a terminal value $S:\mathbb{R}^n\to \mathbb{R}$. We  define the random action functional $J_{t,x}$ as 
\begin{align}\label{random functional}
	J_{t,x}(Z,u )&=\int_{t}^{T} \left(\frac{u^2}{2}(\omega,s) + V(Z_s) \right)\mathrm{d}s+S(Z_T),\\
dZ_s&=u(\omega,s)\mathrm{d}s+\nu^{1/2}dW(s),\quad Z_t=x \text{ and }0\leq t\leq s\leq T,\label{controlled diffusion}
\end{align}
where $\nu$ is a positive constant, $u(\omega,s)$ (denoted by $u(s)$ in rest of the paper) is a control in the set of admissible controls ${\cal U}$, $W$ is a standard Brownian motion in $\mathbb{R}^n$. 
Our main goal is  to find the value process
\begin{align}\label{value process}
U_t(x)=\essinf_{u\in\mathcal{U}}J_{t,x}(Z,u ).
\end{align}
Equivalently, one can characterise the optimal process $u$ that minimizes the functional $J_{t,x}$, when it exists.
The value process $U_t$ is such that  $U_t(x;\omega)\leq J_{t,x}(Z,u;\omega)$ for almost all $\omega \in \Omega$. Contrary to the classical stochastic control problems, for this pathwise problem, one cannot expect to have adapted solutions since the value process at time $t$, $U_t$ depends on $Z_T$. Consequently, we may end up having an optimal control $u$, which is not adapted to the increasing filtration. 

We define ${\cal U}$ as the set of all  measurable processes (not necessarily adapted) such that they are uniformly bounded in $L^\infty(0,T)$, i.e., $\|u(\omega,\cdot)\|_{\infty} \le C$ for a.s. $\omega$ and all $u\in \mathcal{U}$ for some constant $C>0$.

In the next assumption, we present some regularity conditions, which will be necessary throughout the paper. Some of these conditions are needed to guarantee that the optimization problem is well-posed.

\begin{assumption}\label{regularity needed}
We assume the following:
\begin{itemize}
    \item $V$ and $S$ are such that
    $$
    P\left(\int_0^T\vert V(Z_s)\vert ds<\infty\right)=1 \quad\text{and}\quad P\left(\vert S(Z_T)\vert <\infty\right)=1;
    $$ 
    \item $V$ is a bounded Lipschitz map and $S$ is Lipschitz continuous.
\end{itemize}
\end{assumption}
For each $u\in\mathcal{U}$, we have
$$
\vert J_{t,x}(Z,u)\vert\leq \int_{0}^{ T} \left(u^2(s)/2 + \vert V(Z_s)\vert \right)ds+\vert S(Z_T)\vert.
$$
Thus, according to Assumption \ref{regularity needed}, the following result holds true.
\begin{prop}
Let $J_{t,x}$ be the random functional defined in \eqref{random functional} and $u\in{\cal U}$. Then,
$$
P\left(\vert J_{t,x}(Z,u)\vert<\infty\right)=1.
$$
\end{prop}
\begin{lemma}
There exists a sequence $\{u_n(s)\}_{n \in \mathbb{N}} \subset \mathcal{U}$ s.t. the corresponding sequence $J_{t,x}(Z, u_n)$ is decreasing and the limit is $U_t(x)$ as $n$ goes to infinity for a.s. $\omega$.
\end{lemma}
\begin{proof}
    Let $u_1,u_2 \in \mathcal{U}$ and consider the event $A:=\{J(u_1)\le J(u_2)\}$. Define $\hat{u}(\omega,s):=u_1(\omega,s) \mathbbm{1}_A(\omega)+u_2(\omega,s)\mathbbm{1}_{A^c}(\omega)$, observe that
    \begin{align*}
        J(\hat{u})&= J(\hat{u})\mathbbm{1}_A+J(\hat{u})\mathbbm{1}_{A^c}\\
        &= J(u_1)\mathbbm{1}_A+J(u_2)\mathbbm{1}_{A^c}\\
        &=J(u_1) \wedge J(u_2), \qquad \text{a.s}.
    \end{align*}
    Thus the family $\{J_{t,x}(Z, u_n)\}_{n \in \mathbb{N}}$ is directed downwards and one uses the properties of essential infimum \cite[Pg 121]{Neveu} to guarantee an existence of a sequence $\{u_n(\omega,s)\}_{n \in \mathbb{N}} \subset \mathcal{U}$ such that essential infimum in \eqref{value process} becomes a limit, and for the corresponding sequence $J_{t,x}(Z,u_n)$ one can write $J_{t,x}(Z,u_n) \searrow U_t(x)$ a.s. as $n$ goes to infinity. 
    
\end{proof}
\begin{prop}\label{continuity}
The value function $U_t(x)$ defined in \eqref{value process} is continuous on $\mathbb{R}^n \times (0,T)$ a.s. and moreover, it is Lipschitz continuous in spatial variable, uniformly in $t$, and $\alpha$-H\"older continuous for $\alpha <1/2$ in time uniformly in $x$. 
\end{prop}
\begin{proof}
We denote by $Z_s^{t,x,u}$, the solution of \eqref{controlled diffusion}
at time $s$. We first  show the continuity of the value function $U_t(x)$ in spatial variable $x$. We start by picking a sequence $\{u_n(\cdot)\}_{n\in \mathbb{N}}$ 
\color{black}
of admissible controls in $\mathcal{U}$ such that $U_t(x)=\lim\limits_{n\to \infty}J_{t,x}(Z,u_n)$ a.s., i.e.,
\begin{align}\label{inf}
U_t(x)=\lim_{n \to \infty}\int_t^T\left( \frac{u^2_n(s)}{2} + V\left(Z_s^{t,x,u_n(s)}\right) \right)ds+S\left(Z_T^{t,x,u_n(t)}\right).
\end{align}
Consider
\begin{align*}
U_t(x)-U_t(x')
&= \lim_{n \to \infty}\left[\int_t^T\left(V\left(Z_s^{t,x,u_n(s)}\right)-V\left(Z_s^{t,x',u_n(s)} \right)\right)ds\right.\\
&\qquad+\left. S\left(Z_T^{t,x,u_n(t)}\right)-S\left(Z_T^{t,x',u_n(t)}\right)\right].
\end{align*}
Thus,
\begin{align*}
    |U_t(x)-U_t(x')| &\le \lim\limits_{n \to \infty}\left[\int_t^T\left\vert V\left(Z_s^{t,x,u_n(s)}\right)-V\left(Z_s^{t,x',u_n(s)}\right)\right\vert  \,ds \right.\\
 &\quad +\left. \left\vert S\left(Z_T^{t,x,u_n(t)}\right)-S\left(Z_T^{t,x',u_n(t)}\right)\right\vert \right]\\
 &\le \lim\limits_{n \to \infty}\left[\|V'\|_\infty\int_t^T\left\vert Z_s^{t,x,u_n(s)}-Z_s^{t,x',u_n(s)}\right\vert  \,ds \right.\\
 &\quad +\left. \|S'\|_\infty\left\vert Z_T^{t,x,u_n(t)}- Z_T^{t,x',u_n(t)}\right\vert \right].
\end{align*}
Next, we notice from \eqref{controlled diffusion} that $\left\vert Z_s^{t,x,u_n(t)}-Z_s^{t,x',u_n(t)}\right\vert= \vert x-x'\vert$ almost surely, which allows us to conclude that

\begin{align}\label{space}
\vert U_t(x)-U_t(x')\vert\leq C|x-x'|.
\end{align}
Now we prove that $U_t(x)$ is continuous in time variable $t$; we have 
\begin{align}\label{inf-t}
U_{t}(x) &= \lim_{n \to \infty}\left[\int_{t}^{T}\left(\frac{u^2_n(s)}{2} + V(Z_s^{t,x,u_n(s)})\right) ds +S\left(Z_T^{t,x,u_n(t)}\right)\right].
\end{align}
Take $t'>t$ and consider
\begin{align*}
U_{t}(x)&-U_{t'}(x)\\
&= \lim_{n \to \infty}\left[\int_{t'}^{T}\left(V\left(Z_s^{t,x,u_n(s)}\right)-V\left(Z_s^{t',x,u_n(s)}\right) \right)ds+S\left(Z_T^{t,x,u_n(t)}\right)-S\left(Z_T^{t',x,u_n(t)}\right)\right]\\
&~~+\lim_{n \to \infty}\left[\int_{t}^{t'}\left( \frac{u^2_n(s)}{2} + V\left(Z_s^{t,x,u_n(s)}\right) \right)ds \right]. 
\end{align*}

Then, 
\begin{align*}
    |U_t(x)-U_t'(x)|&\le \lim_{n \to \infty}\left[\|V'\|_\infty\int_{t'}^{T}\left\vert Z_s^{t,x,u_n(s)}-Z_s^{t',x,u_n(s)}\right\vert ds\right.\\
&~~+\left. \|S'\|_\infty\left\vert Z_T^{t,x,u_n(t)}- Z_T^{t',x,u_n(t)}\right\vert + C(\|u_n\|^2_\infty+\|V\|_\infty)|t-t'|\right]
\end{align*}
We again notice from \eqref{controlled diffusion} that
\begin{align*}
    &\left\vert Z_s^{t,x,u_n(t)}- Z_s^{t',x,u_n(t)}\right\vert =\left \vert \int^t_{t'}u_n(s)\,ds
    + \int^t_{t'}\sqrt{\nu}\,dW_s \right \vert\\
    &\qquad\le \|u_n\|_\infty|t-t'|+\sqrt{\nu}|W(t)-W(t')|\\
    &\qquad \le C(\sqrt{\nu},W)|t-t'|^\alpha 
\end{align*}
Thus, we get
\begin{align}\label{time}
|U_t(x)- U_{t'}(x)| \le C|t-t'|^\alpha. 
\end{align}
Finally the continuity of $U_t(x)$ in $x$ from \eqref{space} and in $t$ from \eqref{time} implies the joint continuity in $(x,t)$ a.s. as  we have 
\begin{align*}
    |U_t(x)-U_{t'}(x')| & \le |U_t(x)-U_t(x')|+|U_t(x')-U_{t'}(x')|\\ 
    &\le C(|x-x'|+|t-t'|^\alpha).
\end{align*}
\end{proof}

\section{Stochastic HJB equation}\label{Stochastic HJB equation}
In classical stochastic control problems under a Markovian framework, the value function can be represented as a viscosity solution of an HJB equation, which is a deterministic PDE (see for instance \cite{fleming2006controlled}). In our case, using  Bellman's principle, we can prove that the value process is a stochastic viscosity solution to the terminal valued SPDE
\begin{align}\label{Stratonovich}
\begin{cases}
     dv(s,x)= -\sqrt{\nu}\nabla v(s,x)\circ dW(s)+\left( V(x)-\frac{|\nabla v|^2}{2}\right)\mathrm{d}s, \quad &\text{in } (0,T)\times \mathbb{R}^n,\\
    v(T,x)=S(x), & \text{on } \mathbb{R}^n.
    \end{cases}
\end{align}
Taking into account the definition of stochastic viscosity solutions and the simpler form of Stratonovich expressions 
it is more convenient to write the  SPDE in the Stratonovich sense as above. In It\^o form it reads
\begin{align}\label{SPDE value process}
\begin{cases}
     dv(s,x)= -\sqrt{\nu}\nabla v(s,x).dW(s)\\ \hspace{3cm}+\left(V(x)-\frac{|\nabla v|^2}{2}+\frac{\nu}{2}
     {\mathbb D}_s (\nabla v)(s,x)\right)\mathrm{d}s, \quad &\text{in } (0,T)\times \mathbb{R}^n, \\
    v(T,x)=S(x), & \text{on } \mathbb{R}^n.
    \end{cases}
\end{align}
Let us consider the following equation
 \begin{align}\label{stochastic Hamilton-Jacobi equation}
 \begin{cases}
      d\Phi(s,x)= -\sqrt{\nu}\nabla \Phi(s,x)\circ dW(s) , \quad &\text{in } (0,T)\times \mathbb{R}^n,\\
     \Phi(T,x)=\phi(x), & \text{on } \mathbb{R}^n,
     \end{cases}
 \end{align}

\begin{lemma}\label{solution}
Let $\phi\in C^2_b(\mathbb{R}^n)$, then $\Phi(s,x):=\phi(x +\sqrt{\nu} (W_T-W_{s}))$ is a classical solution to \eqref{stochastic Hamilton-Jacobi equation}, for all $s\leq T$.
\end{lemma}
We recall the definition of pathwise viscosity solution for the following first order initial value problem
\begin{align}\label{HJB}
\begin{cases}
     du(s,x)= H(Du)\circ dW(s)+ F(Du,x)\,ds,\quad &\text{in } (0,T)\times \mathbb{R}^n,\\
    u(T,x)=f(x), & \text{on } \mathbb{R}^n.
    \end{cases}
\end{align}
where $F: \mathbb{R}^n \times \mathbb{R}^n \to \mathbb{R},\, H:\mathbb{R}^n \to \mathbb{R}^n$ and $f:\mathbb{R}^n \to \mathbb{R}$ are given functions and we consider the Stratonovich integral in the generalized non-adapted sense. Let $\Phi$ denote the solution of corresponding stochastic Hamilton-Jacobi equation, i.e., $d \Phi = H(D \Phi) \circ dW(s)$ in $(0,T) \times \mathbb{R}^n$ with $\Phi(T,x)=f(x)$ on $\mathbb{R}^n$.
\begin{defi}\label{viscosity}
An upper semi-continuous (resp. a lower semi-continuous) function $v$ defined on $[0,T] \times \mathbb{R}^n$ is said to be a viscosity sub-solution (resp. super-solution) to \eqref{HJB} if it is bounded from above (resp. from below) with terminal data satisfying $u(\cdot, T) \le f(x)$ (resp. $u(\cdot,T) \ge f(x)$), and, whenever $\phi \in C^2_b(\mathbb{R}^n)$, $h=h(\phi)>0$, $g\in C^1([0,T]), \Phi(s,x)\in C^2_b(\mathbb{R}^n)$,  for $s\in (s_0-h,s_0+h)$, and the map $(s,x)\mapsto v(s,x)-\Phi(s,x)-g(s)$ attains a local maximum (resp. local minimum) at $(s_0,x_0) \in \mathbb{R}^n \times (s_0-h,s_0+h)$, then 
\begin{align}\label{condition}
    -g'(s_0) \le F(D \Phi(s_0,x_0),x_0) \quad \left(\text{resp. }-g'(s_0) \ge F(D \Phi(s_0,x_0),x_0)\right). 
\end{align} 
\end{defi}

\begin{rem}\cite[Lemma 3.4]{seeger2018perron}
Suppose a function $v$ satisfies the hypothesis of Definition \ref{viscosity} such that \eqref{condition} is true when we replace local maximum (resp. local minimum) with local strict maximum (resp. local strict minimum), then $v$ is a viscosity sub-solution (resp. super-solution) in the sense of Definition \ref{viscosity}.
\end{rem}
In order to establish that the value function is a solution to the stochastic HJB in the viscosity sense, we next show that the value function satisfies the following pathwise Dynamic Programming Principle.
\begin{prop}[Bellman's Optimality Principle]\label{prop Bellman principle}
Let $\delta$ be a constant such that $t+\delta \le T$ and $\tilde{U}_t$ be such that
$$
\tilde{U}_t(x)=\essinf_{u\in\mathcal{U}}\left\{\int_{t}^{t+\delta} \left(\frac{u^2(s)}{2} + V(Z_s) \right)ds\right\}+U_{t+\delta}(Z_{t+\delta}).
$$
Then, $\tilde{U}_t={U}_t$.
\end{prop}
\begin{proof}
We observe that
\begin{align*}
	J_{t,x}(Z,u )&=\int_t^T \left( \frac{u^2(s)}{2} + V(Z_s) \right)ds+S(Z_T)\\
&=\int_{t}^{t+\delta} \left( \frac{u^2(s)}{2} + V(Z_s) \right)ds+\underbrace{\int^{T}_{t+\delta}\left( \frac{u^2(s)}{2} + V(Z_s) \right)ds+S(Z_T)}_{\coloneqq J_{t+\delta,Z_{t+\delta}}(Z,u)}\\
	&\geq \int_t^{ t+\delta} \left(\frac {u^2(s)}{2} + V(Z_s) \right)ds+	U_{t+\delta}(Z_{t+\delta}).
\end{align*}
Since this is true for any $u\in\mathcal{U}$, we get 
\begin{align}\label{BP-one}
{U}_t(x)\geq \essinf_{u\in\mathcal{U}}\left\{\int^{t+\delta}_t \left(\frac{ u^2(s)}{2} + V(Z_s) \right)ds\right\}+U_{t+\delta}(Z_{t+\delta}).
\end{align}
Let $\{u_n(s)\}_n$ be the sequence as in Prop \ref{continuity} and define
\begin{align*}
\tilde{u}(s)&=\begin{cases}
u(s),&0\leq s \leq t+\delta\\
u_n(s),&t+\delta\leq s\leq T.
\end{cases}
\end{align*}
Then we have  
\begin{align*}
    \int^{t+\delta}_t \left(\frac{u^2(s)}{2} + V(Z_s) \right)ds+J_{t+\delta,x}(Z_{t+\delta},u_n)&=\int_t^T \left(\frac{\widetilde{u}^2(s)}{2} + V(Z^{\widetilde u}_s) \right)ds+S(Z_T)\\
    &\geq {U}_t(x),
\end{align*}
for all $\tilde{u}\in \mathcal{U}$, now first taking limit as $n$ goes to infinity and then taking infinimum over the family $\mathcal{U}$ gives 
\begin{align}\label{BP-two}
\essinf_{u\in\mathcal{U}}\left\{\int^{t+\delta}_t \left(\frac{ u^2(s)}{2} + V(Z_s) \right)ds\right\}+U_{t+\delta}(Z_{t+\delta})\geq U_t(x).
\end{align}
Combining \eqref{BP-one} and \eqref{BP-two} gives the desired result.
\end{proof}

\begin{prop}[Existence of viscosity solution]
The value process $U$ defined by \eqref{value process}  is a viscosity solution to the terminal value SPDE \eqref{Stratonovich}, in the sense of Definition \ref{viscosity}.
\end{prop}
\begin{proof}
The continuity of the value function $U$ is established in Proposition \ref{continuity}.
Next we show that if $(U_t)_{t>0}$ satisfies the hypothesis of Definition \ref{viscosity}, then it satisfies \eqref{condition}. For that purpose, we denote by $\Phi(s,x)$ a smooth solution of the following equation 
\begin{align}\label{HJ}
\begin{cases}
     dv(s,x)= -\sqrt{\nu}\nabla v(s,x)\circ dW(t), \quad & \text{in } (0,T) \times \mathbb{R}^n,\\
    v(T,x)=\phi(x), & \text{in } \mathbb{R}^n,
    \end{cases}
\end{align}
Let $(t_0,x_0)\in (0,T) \times \mathbb{R}^n $ be such that the map  $(s,x)\mapsto U(s,x)-\Phi(s,x)-g(s)$ attains a local minima at $(t_0,x_0)$. Consider a ball $B_\varepsilon$ of radius $\varepsilon>0$ centered at $x_0$ and, for some $h>0$, let $(t_n,x_n)\in [t_0-h,t_0+h] \times B_\varepsilon$ be a sequence such that $(t_n,x_n)$ converges to $(t_0,x_0)$ as $n$ goes to infinity. Define a stopping time $\theta_n= \inf\{s \ge t_n \,: (s,Z_s^{t_n,x_n}) \notin [t_n-h,t_n+h] \times B_\varepsilon\}$; then $\theta_n$ converges to $t_0$ as $n$ goes to infinity. Applying It\^{o} formula for non adapted processes to the process $Z_t$ for the map $\Phi(s,x)$, we get
\begin{align*}
\Phi(\theta_n,Z_{\theta_n})-\Phi(t_n, Z_{t_n}) &= \int^{\theta_n}_{t_n}\left( -\sqrt{\nu} \nabla \Phi+ \sqrt{\nu}\nabla \Phi\right)(s,Z_s) \,\circ dW(s)\\
    &+\int^{\theta_n}_{t_n} (u\cdot \nabla)\Phi\,ds\\
    &= \int^{\theta_n}_{t_n} (u\cdot \nabla)\Phi\,ds. 
\end{align*}
With this observation and the Bellman's optimality principle we have
\begin{align*}
    U_{t_n}(x_n)=\essinf_{u\in\mathcal{U}}&\left\{\int_{t_n}^{\theta_n} \left(\frac{u^2(s)}{2} + V(Z_s) \right)ds+U_{\theta_n}(Z_{\theta_n})\right\}\\
    \ge \essinf_{u\in\mathcal{U}}&\bigg\{\int_{t_n}^{\theta_n} \left(\frac{u^2(s)}{2} + V(Z_s) \right)ds+\Phi(\theta_n,Z_{\theta_n})\\
    &\quad +g(\theta_n)+U_{t_n}(x_n)-\Phi(t_n,Z_{t_n})-g(t_n)\bigg\}.
\end{align*}
This gives 
\begin{align*}
    0 \ge \essinf_{u \in \mathcal{U}} \left\{\int_{t_n}^{\theta_n} \left(\frac{u^2(s)}{2} + V(Z_s) \right)ds+\int^{\theta_n}_{t_n} (u \cdot \nabla) \Phi(s,Z_s)\mathrm{d}s\right\}+\int^{\theta_n}_{t_n}g'(s)\mathrm{d}s.
\end{align*}
Dividing both sides by $\theta_n-t_n$ and letting $n$ go to infinity to get
\begin{align*}
    -g'(t_0) &\ge \essinf_{u\in\mathcal{U}}\left\{\frac{u^2(t_0)}{2} +(u\cdot \nabla) \Phi(t_0,Z_{t_0}) \right\}+ V(Z_{t_0})\\
    &=-\frac{|\Phi(t_0,Z_{t_0})|^2}{2}+V(Z_{t_0}).
\end{align*}
and therefore $U$ is a viscosity super-solution. Analogously it can be shown that $U$ is a viscosity sub-solution and this concludes the proof.
\end{proof}

Next we shall prove the comparison result for the viscosity solutions of \eqref{Stratonovich}, where we closely follow the arguments of \cite{seeger2018perron}. We first present a technical lemma.
\begin{lemma}
    Let $u$ and $v$ be viscosity sub-solution and super-solution of \eqref{Stratonovich} respectively, then $w(x,y,z):= u(x,s)-v(y,s)$ is a viscosity sub-solution of 
    \begin{align}\label{doubled}
&    dw(x,y,s) = -\sqrt{\nu}( \nabla_x w + \nabla_y w)\circ \,dW(s) \notag\\ &\qquad+V(x)-V(y) -\left(\frac{|\nabla_x w|^2}{2}-\frac{|\nabla_y w|^2}{2}\right)
\end{align}
where $(x,y,s) \in (B_R(0))^2 \times (0,T)$ and $w(x,y,s)=0$ on $(B_R(0))^2 \times (0,T)$. 
\end{lemma}
\begin{proof}
    Let  $\psi \in C^1(0,T)$ such that 
    \begin{align}\label{zee}
    z(x,y,s):=w(x,y,s)- \Phi(x,y,s)-\psi(s)
    \end{align}
    attains a strict maximum at $(x_0,y_0,s_0)\in (B_R(0))^2 \times (0,T)$, where $\Phi(x,y,s)$  solves 
\begin{align}\label{noise}
d\Phi = -\sqrt{\nu}( \nabla_x \Phi + \nabla_y \Phi) \circ dW(s).
\end{align}
In fact with arguments similar to the ones used in the proof of Proposition \ref{solution}, one concludes that the classical solution of \eqref{noise} is given by $\Phi(x,y,s)=\Phi(X^{-1}_s, Y^{-1}_s,T)$ where $X^{-1}_s= x+\sqrt{\nu}(W_T-W_s), Y^{-1}_s= y+ \sqrt{\nu}(W_T-W_s)$ and $\Phi(\cdot,\cdot, T)$ denotes the final data defined on $(B_R(0))^2$.\\
For a fixed $0<\delta <1$, consider the quantity
\begin{align}\label{doubling}
\Psi_\delta(x,y,s,r):=u(x,s)- v(y,r)- \Phi(X^{-1}_s,Y^{-1}_r,T)- \psi(s)- \frac{|s-r|^2}{2\delta}
\end{align} 
in the compact set $\Omega:= \overline{B}(x_0)\times \overline{B}(y_0) \times [s_0-h,s_0+h]^2$, for some $h>0$, where $\overline{B}(x_0)$ and $\overline{B}(y_0)$ denote balls of sufficiently small radius around $x_0$ and $y_0$. Assume that $\Psi_\delta$ attains a maximum at a point $(x_\delta, y_\delta, s_\delta, r_\delta)\in \Omega$ (guaranteed as  $u-v$ is upper semi-continuous). First notice that since we are in a compact domain $\Omega$, there exists $M>0$ such that $\Psi_\delta(x,y,s,r)+\frac{|s-r|^2}{2\delta} \le M$; then we must have $\frac{|s_\delta-r_\delta|^2}{2\delta} \le 2M$ for all $\delta>0$, because if there exists some $\delta >0$ such that $\frac{|s_\delta-r_\delta|^2}{2\delta}> 2M$, then $\Psi_\delta(x_\delta, y_\delta, s_\delta, r_\delta) < -M < \Psi_\delta(x_\delta,y_\delta,s_0,s_0)$, contradicting the fact that $(x_\delta, y_\delta, s_\delta, r_\delta)$ is a maximum. Thus we conclude that $\lim\limits_{\delta \to 0}|s_\delta -r_\delta| =0$.
Next we define $\Psi_1(x,y,s,r):= u(x,s)-v(y,r)-\Phi(X^{-1}_s,Y^{-1}_r,t)-\psi(s) -|s-r|^2$ and show that  $\lim_{\delta \to 0} \Psi_1(x_\delta,y_\delta,s_\delta,r_\delta)= z(x_0,y_0,s_0)$, where $z$ is defined in \eqref{zee}. To do this we write 
\begin{align*}
    \Psi_1(x,y,s,r) &= u(x,s) -v(y,s)- \Phi(X^{-1}_s,Y^{-1}_s,t)-\psi(s) - |s-r|^2\\
    &+(v(y,s)-v(y,r))+(\Phi(X^{-1}_s,Y^{-1}_s,t)-\Phi(X^{-1}_s,Y^{-1}_r,t))\\
    &< u(x_0,s_0)-v(y_0,s_0) - \Phi(X^{-1}_{s_0}, Y^{-1}_{s_0},t)-\psi(s_0)\\
    &+(v(y,s)-v(y,r))+(\Phi(X^{-1}_s,Y^{-1}_s,t)-\Phi(X^{-1}_s,Y^{-1}_r,t))
\end{align*}
where in the last inequality we have used the fact that $z(x,y,s)$ attains a strict maximum at $(x_0,y_0,s_0)$.We now use the fact that $s_\delta -r_\delta \to 0$ as $\delta \to 0$, the upper semi-continuity of $v$ and the continuity of $\Phi$ to conclude that $$\limsup_{\delta\to 0} \Psi_1(x_\delta,y_\delta,s_\delta,r_\delta) \le u(x_0,s_0)-v(y_0,s_0) - \Phi(X^{-1}_{s_0}, Y^{-1}_{s_0},t)-\psi(s_0).$$
On the other hand, for small $\delta >0$ we have $\Psi_1(x_\delta, y_\delta,s_\delta, r_\delta)>\Psi_\delta(x_\delta, y_\delta,s_\delta, r_\delta) \ge \Psi_\delta(x_0,y_0,s_0,s_0)$ $=z(x_0,y_0,s_0)$ implying that $\liminf\limits_{\delta \to 0}\Psi_1(x_\delta,s_\delta,y_\delta,r_\delta) \ge  z(x_0,y_0,s_0)$ and thus we can conclude that
$\lim\limits_{\delta \to 0}\Psi_1(x_\delta,s_\delta,y_\delta,r_\delta)$ exists and equals $z(x_0,y_0,s_0).$

We are now in a position to claim that the sequence $(x_\delta, y_\delta, s_\delta, r_\delta)$ converges to $(x_0,y_0,s_0,s_0)$ as $\delta$ goes to zero. To do this it suffices to show that any arbitrary sub-sequence $(x_{\delta_k}, y_{\delta_k}, s_{\delta_k}, r_{\delta_k})$ of the original sequence $(x_\delta, y_\delta, s_\delta, r_\delta)$ converges to $(x_0,y_0,s_0,s_0)$ since $(x_\delta, y_\delta, s_\delta, r_\delta)$ not converging (in a compact set) would mean existence of two sub-sequences converging to two different limits. If possible, let $(x_{\delta_k}, y_{\delta_k}, s_{\delta_k}, r_{\delta_k})$ converge to $(\tilde x_0, \tilde y_0, \tilde s_0, \tilde s_0)$; then from the upper semi-continuity of $\Psi_1$ we have  $\lim\limits_{k \to \infty} \Psi_1(x_{\delta_k}, y_{\delta_k}, s_{\delta_k}, r_{\delta_k})\le  \Psi_1(\tilde x_0, \tilde y_0, \tilde s_0, \tilde s_0)=z(\tilde x_0,\tilde y_0,\tilde s_0)$ and also from the preceding arguments we know that $\lim\limits_{k \to \infty} \Psi_1(x_{\delta_k}, y_{\delta_k}, s_{\delta_k}, r_{\delta_k})=z(x_0,y_0,s_0)$.
This implies that $z(\tilde x_0,\tilde y_0,\tilde s_0)\ge z(x_0,y_0,s_0)$. But, as $(x_0,y_0,s_0)$ is a strict maximum of $z(x,y,s)$, we must have $(\tilde x_0,\tilde y_0,\tilde s_0)=(x_0,y_0,s_0)$. Therefore the original sequence $(x_\delta,y_\delta,s_\delta,r_\delta)$ converges to $(x_0,y_0,s_0,s_0)$ as $\delta$ goes to zero, which implies that for small $\delta$, we can assume that $(x_\delta,y_\delta,s_\delta,r_\delta)$ lies in the open set $\Omega^{\mathrm{o}}:= B(x_0)\times B(y_0) \times (s_0-h,s_0+h)^2$.

We now use the fact that $u(x,s)$ is a viscosity sub-solution and note that the map $$\psi^u:(x,s) \mapsto v(y_\delta,r_\delta)+ \Phi(X^{-1}_s,Y^{-1}_{r_\delta},T)+\psi(s)+\frac{|s-r_\delta|^2}{2\delta}$$  is a test function such that $u(x,s)-\psi^u$ attains a maximum at $(x_\delta,s_\delta)$ from $\eqref{doubling}$; therefore we get

\begin{align}\label{sub}
-\left(\psi'(s_\delta) +\frac{(s_\delta - t_\delta)}{\delta}\right) \le V(x_\delta) - \frac{|\nabla \Phi(X^{-1}_{s_\delta},Y^{-1}_{r_\delta},T)|^2}{2}.
\end{align}
Similarly the map 
$$\psi^v:(y,r) \mapsto u(x_\delta,s_\delta)- \Phi(X^{-1}_{s_\delta},Y^{-1}_r,T)-\psi(s_\delta)-\frac{|s_\delta-r|^2}{2\delta} $$ is a test function such that
$v(y,r)-\psi^v$ attains a minimum at $(y_\delta,r_\delta)$ from \eqref{doubling}; using the fact that $v(y,r)$ is a viscosity super-solution we get 
\begin{align}\label{super}
    -\frac{s_\delta-r_\delta}{\delta} \ge V(y_\delta)- \frac{|\nabla \Phi(X^{-1}_{s_\delta},Y^{-1}_{r_\delta},T)|^2}{2}.
\end{align}
Combining $\eqref{sub}$ and $\eqref{super}$, we have 
\begin{align*}
   - \psi'(s_\delta) &\le V(x_\delta)-V(y_\delta)-\left(\frac{|\nabla_x \Phi(X^{-1}_{s_\delta},Y^{-1}_{r_\delta},T)|^2}{2} -\frac{|\nabla_y \Phi(X^{-1}_{s_\delta},Y^{-1}_{r_\delta},T)|^2}{2}\right)
\end{align*}
Letting $\delta$ go to zero, using continuities of $\psi$, $V$ and $\nabla \Phi$, we get
$$ -\psi'(s_0) \le (V(x_0)-V(y_0)) -\frac12 \left(| \nabla_x \Phi(x_0,y_0,s_0)|^2-| \nabla_y \Phi(x_0,y_0,s_0)|^2\right).$$
Therefore we conclude that  $w(x,y,s)=u(x,s)-v(y,s)$ is a viscosity sub-solution of \eqref{doubled}.
\end{proof}

\begin{prop}[Comparison principle for viscosity solutions]\label{compare}
Let $u$ and $v$ be viscosity sub-solution and super-solution respectively of \eqref{Stratonovich}. Then 
$$\sup_{x \in B_R(0)}( u(x,s)-v(x,s))_+= \sup_{x \in B_R(0)}(u(x,0)-v(x,0))_+$$ for all $s\in [0,T]$.
\begin{proof}
 We first note that it suffices to prove the comparison principle for the case when $u(x,0)\le v(x,0)$ for all $x\in B_R(0)$. Otherwise we work with $v(x,s)+c$ instead of $v(x,s)$ for some constant $c>0$  (for e.g. $c= \|S'\|_\infty$) and make use of the fact that $v+c$ is a viscosity super-solution of \eqref{Stratonovich} as well because $F$ in the Definition \ref{viscosity} is independent of the solution $u$. Therefore it suffices to prove that $\sup_{x \in B_R(0)}(u(x,s)-v(x,s))_+ = 0$ for all $s\in [0,T]$. On the contrary, suppose there exists $s_0\in [0,T]$ such that $\sup_{x \in B_R(0)}(u(x,s_0)-v(x,s_0))_+ > 0$. Choose $\mu>0$ small such that $\sup\limits_{x \in B_R(0)}( u(x,s_0)-v(x,s_0))> \mu s_0$. Note that the map $\Phi_\epsilon(x,y,s):= \frac{1}{2\epsilon}|x-y|^2$ is a classical solution of \eqref{noise},  consider the quantity $$u(x,s)-v(y,s)- \frac{1}{2\epsilon}|x-y|^2-\mu s$$
and let $(x_\epsilon, y_\epsilon, s_\epsilon)$ be a point of maxima. Then we use the fact that $u-v$ is a viscosity sub-solution of \eqref{doubled}  to conclude that 
\begin{align*}
\mu \le V(x_\epsilon)-V(y_\epsilon);    
\end{align*}
letting $\epsilon$ go to zero and using the fact that $\lim\limits_{\epsilon \to 0}(x_\epsilon-y_\epsilon)=0$ (similar to the proof of $\lim\limits_{\delta \to 0}(s_\delta -r_\delta)=0$), we deduce that $\mu \le 0$, which contradicts the choice of $\mu$.
\end{proof}
\end{prop}

\begin{corollary} There exists a unique viscosity solution of \eqref{Stratonovich} in the sense of Definition \ref{viscosity}.
\begin{proof}
Let $u$ and $v$ be two viscosity solutions of \eqref{Stratonovich} in the sense of Definition \ref{viscosity}. From Proposition \ref{compare}, we get $\sup\limits_{x \in B_R(0)}|u(x,s)-v(x,s)| \le 0$ for all $s\in[0,T]$, which implies the uniqueness of viscosity solution for \eqref{Stratonovich} on $B_R(0)$. Since $R>0$ is arbitrary, we have the uniqueness of viscosity solutions in $\mathbb{R}^n$. 
\end{proof}
\end{corollary}

\begin{rem}
    The proofs in this section can be extended to a general SPDE 
    \begin{align}
\begin{cases}\label{generic}
     dv(s,x)= -\sqrt{\nu}\nabla v(s,x)\circ dW(s)+\left( V(x)+\essinf\limits_{u \in \mathcal{U}}\left(L(u)+u\cdot \nabla v\right)\right)\mathrm{d}s, \quad &\text{in } (0,T)\times \mathbb{R}^n,\\
    v(T,x)=S(x), & \text{on } \mathbb{R}^n.
    \end{cases}
\end{align}
where $L$ is a Lipschitz map. Notice that \eqref{generic} reduces to \eqref{Stratonovich} for the case when $L(u)=\frac{u^2}{2}$. One way to extend the results to the case $L(u)$ is to impose the condition of equicontinuity in the time variable on the set of admissible controls $\mathcal{U}$. This SPDE corresponds to the control problem where the cost function and state equation is given by 
\begin{align*}
	J_{t,x}(Z,u )&=\int_{t}^{T} \left(\frac{u^2}{2}(\omega,s) + V(Z_s) \right)\mathrm{d}s+S(Z_T),\\
dZ_s&=u(\omega,s)\mathrm{d}s+\nu^{1/2}dW(s),\quad Z_t=x \text{ and }0\leq t\leq s\leq T,
\end{align*}
\end{rem}

\begin{rem}
    Our results are also extendable to more generic state equations of the form  
$$dZ_s=u(\omega,s)\mathrm{d}s+\sqrt{\nu}\theta(Z_s)dW(s),\quad Z_t=x \text{ and }0\leq t\leq s\leq T$$
and in this case the value function will satisfy the following SPDE
\begin{align*}
\begin{cases}
     dv(s,x)= -\sqrt{\nu}\theta(x)\nabla v(s,x)\circ dW(s)+\left( V(x)-\frac{|\nabla u|^2}{2} \right)\mathrm{d}s, \quad &\text{in } (0,T)\times \mathbb{R}^n,\\
    v(T,x)=S(x), & \text{on } \mathbb{R}^n.
    \end{cases}
\end{align*}
\end{rem}

\section{On the optimal drift}\label{characterization}
\subsection{Characterization}
In this section, we present a derivation of the SPDE satisfied by the optimal drift of the process using stochastic calculus of variations, assuming the optimal control is attained and is of the form 
$u (\omega ,s)= u^* (\omega, Z_s)$, where $u^*$ is a non necessarily adapted process.
We have worked in a larger class of controls and it is an open question weather this process $u^*$ is well defined.\color{black}

Let the process $Z$, which satisfies the SDE $dZ_s =u^*_s (Z_s )ds +\nu^{\frac{1}{2}} dW_s, ~s\in [t,T], Z_t=x$, be a critical process for the action functional $J$ given by \eqref{random functional}. We have, for every stochastic process $h$ differentiable in time, with time derivative in $L^2$ and such that $h(t)=0$,
$$\left.\frac{d}{d\varepsilon}\right|_{\varepsilon=0} J(Z+\varepsilon h) = \int_t^T u_s^* (Z_s ) .\dot  h(s) ds +\int_t^T \nabla V(Z_s). h(s) ds  +\nabla S(Z_T ). h(T).$$
 As $h$ is of bounded variation, then
$$d\left(u^*_s (Z_s )). h(s)\right)= d\left(u^*_s (Z_s) \right). h(s) + u^*_s (Z_s ).\dot h(s) ds,$$ and since $h(t)=0$, we have, for all such $h$, 
$$\left.\frac{d}{d\varepsilon}\right|_{\varepsilon=0} J(Z+\varepsilon h) = \int_t^T h(s). \left[-d(u^*_s (Z_s ))+\nabla V(Z_s)\right]ds +\left[\nabla S(Z_T )+u^*_T (Z_T )\right].h(T)$$
Therefore the following equation holds (a.e.) for the critical process $Z$:
\begin{equation}\label{sde u}
d\left(u^*_s (Z_s )\right)=\nabla V(Z_s )ds, \quad u^*_T  (Z_T )=-\nabla S (Z_T).
\end{equation}

Let  $u^*$ be of the form 
 $du^* = \sum _j X_j (u^* ) \circ dW_t^j +X_0 (u^* ) dt$, where the Stratonovich integral should be interpreted in the non-adapted sense. By  It\^o-Wentzell's formula \eqref{Ito-Wentzell formula} 
we have,
\begin{align*}
	d u^*_s  (Z_s )&= \sum_j \left[X_j (u_s^* )+\nu^{\frac{1}{2}} \partial_j u^*_s \right] ( Z_s) \circ dW_s^j \\
	&+\left[X_0 (u^* )  +(u_s^* .\nabla ) u^*_s \right] (Z_s )ds
\end{align*}
From Equation \eqref{sde u} we deduce that $X_j (u^* )= -\nu^{\frac{1}{2}}\partial_j  (u^* ) $ and, a posteriori, that
$X_0 (u^* )=  -(u^* .\nabla )u^* +\nabla V $. This means that the drift of the stochastic control problem satisfies the non-adapted SPDE
\begin{align}\label{drift}
	du^*_s (x)  =-\nu^{\frac{1}{2}}\nabla  u^*_s (x)\circ dW_s -\left(  (u^*_s.\nabla )u^*_s   -\nabla V \right)(x)ds
\end{align}
for all $s\in [t,T]$, with boundary condition $u^*_T (x) =-\nabla S(x)$.

\begin{rem}
Backward SPDEs such as \eqref{drift}  also appear in the study of doubly stochastic differential equations \cite{PP}. 
\end{rem}

\subsection{Conserved quantities}\label{symmetries}

It is well known that Noether-type theorems are important in Physics, but also in the construction of numerical methods that preserve symmetries, for example.
Conserved quantities for   stochastic action functionals defined as the expectation of our pathwise action $J$ were studied in \cite{Zambrini} and \cite{Huang}, in particular.  In that context, the 
 corresponding constants of motion are martingales and are related to invariants of the associated deterministic Hamilton-Jacobi-Bellman equation.

 In this subsection, we initiate the characterization of space-time transformations that leave our action functional invariant.
 \color{black}
 
 Consider a smooth (possibly random) vector field $Y: ]t,T[\times \mathbb R^n \rightarrow [t,T]\times \mathbb R^n$ of the form
 $Y(s,x)= (T(s), X(s,x))$. Denote by $\Phi_{\epsilon} =(\varphi_{\epsilon}^0 , \varphi_{\epsilon})$ the flow generated by $Y$. In particular,
 
 $$\left.\frac{d}{d\epsilon}\right|_{\epsilon=0} \varphi_{\epsilon}^0 (t) = T(t) , \quad \left.\frac{d}{d\epsilon}\right|_{\epsilon=0}\varphi_{\epsilon} (t,x)= X(t,x).$$
In order to stay closer to the intuition of  the classical counterparts of Lagrangian symmetries, we shall denote, for a stochastic process of the form $d\xi_t =dM_t + \eta_t dt$, its bounded variation part by $D_t \xi_t =\eta_t$. 

  \begin{defi}
  A vector field $Y$ as above is called a Lagrangian infinitesimal variation symmetry of the action functional $J$ if its flow is conserved in the sense that if, for every $ t_1 ,t_2 \in [t,T], t_1 <t_2 $ and every $\epsilon >0$, we have, almost-surely,
\begin{align}\label{VS}
 \int_{t_1}^{t_2} \Big( \frac{1}{2} | D(Z_s)|^2 +V(Z_s) \Big) \mathrm{d}s =\int_{\varphi_{\varepsilon}^0 (t_1 )}^{\varphi_{\varepsilon}^0 (t_2 )}
  \Big( \frac{1}{2} | D (\varphi_{\varepsilon} (Z_{(\varphi_{\varepsilon}^0 )^{-1}(s)}))|^2 +V(\varphi_{\varepsilon}(Z_{(\varphi_{\varepsilon}^0 )^{-1}(s)}))\mathrm{d}s.
  \end{align}
  \end{defi}
 Consider  a vector field $Y$ as above. If $Y$ is an infinitesimal variation symmetry of the action functional $J$,
  by derivating \eqref{VS}  in $\varepsilon$ at $\varepsilon =0$, we obtain, for every $t_1 ,t_2$, the following equation, that characterizes symmetries of the action:
\begin{align}\label{derivative}
 \left(\langle DZ, DX \rangle - \left(\frac{1}{2} |DZ |^2 -V\right) \dot T +\langle \nabla V, X\rangle \right) (s, Z_s )=0
\end{align}
the equality holding  almost surely.

If $u^*$ is the drift of the minimising process, since $d\mathrm{u^*} (s, Z_s )=\nabla V(Z_s)ds$, we have
 $$D( \langle X, u^* \rangle )(s, Z_s ) =(  \langle DX,u^* \rangle +\langle X,\nabla V \rangle ) (s, Z_s ).$$
 On the other hand,
 $$D \left(\frac{1}{2} |u^* |^2\right)(s, Z_s) =\langle u^*, \nabla V \rangle (s,Z_s ); $$
 therefore, using \eqref{derivative}, we have
 $$ D\left( \langle X,u^*\rangle -T \left(\frac{1}{2} |u^* |^2 -V\right)\right) (s,Z_s )=-T (\langle u^* ,\nabla V \rangle -DV) (s,Z_s ) $$
and, finally,
$$ D\left( \langle X,u^* \rangle -T \left(\frac{1}{2} |u^* |^2 -V\right)\right) (s,Z_s ) =\frac{\nu}{2} T (s) \Delta V(Z_s ). $$
For infinitesimal variations symmetries where $T=0$, or for action functionals with harmonic potential functions, we obtain the following associated ``conserved quantities", in the spirit of \cite{Zambrini} and \cite{Huang} (namely stochastic processes with vanishing bounded variation part): 
$$\left( \langle X,u^* \rangle -T \left(\frac{1}{2} |u^* |^2 -V\right)\right) (s,Z_s ) .$$
An obvious example is derived from time translation, namely $T=1, X=0$, when $V$ is harmonic. Another example is $V$=0, $T=0$ and 
$X(x)=Rx$, where $R$ is a rotation matrix.

It will be interesting  to relate these invariant quantities to the symmetries of Hamilton-Jacobi-Bellman stochastic equations as it is done for  deterministic PDEs.

\section{Data availability}
We do not consider data in the analysis we perform in this paper, because we address a given problem with a theoretical approach. All the material needed to understand this paper can be found in the references.

\section{Acknowledgements}
The first and second authors acknowledge the support of the FCT project UIDB/00208/2020. The third author would like to thank the FCT project CEMAPRE/REM-UIDB/05069/2020.

\bibliographystyle{plain}
\bibliography{Pathwise-OC}
\end{document}